\documentclass[a4paper,11pt,reqno]{amsart}
\usepackage{amssymb}
\usepackage{enumerate}
\usepackage{amscd,eucal}
\usepackage{graphicx}
\newtheorem{theorem}{Theorem}[section]
\newtheorem{lemma}[theorem]{Lemma}
\newtheorem{proposition}[theorem]{Proposition}

\newtheorem{theorem-definition}[theorem]{Theorem-Definition}

\theoremstyle{definition}
\newtheorem{definition}[theorem]{Definition}

\newtheorem{problem}[theorem]{Problem}
\newtheorem{question}[theorem]{Question}
\newtheorem{example}[theorem]{Example}

\theoremstyle{remark}
\newtheorem{remark}[theorem]{Remark}

\numberwithin{equation}{section}

\allowdisplaybreaks

\def\aa{{\mathfrak{B}}}

\def\dd{{\mathcal{D}}}

\def\hc{{\mathfrak{H}}}

\def\m2{{{\widetilde{\mathfrak{m}}}_2}}
\def\mb{{\boldsymbol{l}}}
\def\nb{{\boldsymbol{m}}}
\def\nn{{\mathcal{N}}}

\def\sss{{\mathcal{S}}}

\DeclareFontEncoding{OT2}{}{}
\DeclareFontSubstitution{OT2}{cmr}{m}{n}
\DeclareSymbolFont{cyss}{OT2}{wncyss}{m}{n}
\DeclareMathSymbol{\sh}{\mathbin}{cyss}{`x}

\allowdisplaybreaks
\begin{document}

\baselineskip 16pt 

\title[Desingularization of multiple zeta-functions]{Desingularization of complex multiple zeta-functions}


\author[H. Furusho, Y. Komori, K. Matsumoto, and H. Tsumura]{Hidekazu Furusho, Yasushi Komori, Kohji Matsumoto, and Hirofumi Tsumura }


\subjclass[2010]{Primary 11M32; Secondary 11M41}
\keywords{Complex multiple zeta-function, desingularization, 
multiple Bernoulli numbers.
}
\thanks{Research of the authors
supported by Grants-in-Aid for Science Research (no. 24684001 for HF,
no. 25400026 for YK, no. 25287002 for KM, no. 15K04788 for HT, respectively), JSPS}
\maketitle

\begin{abstract}
We introduce the method of desingularization of multi-variable multiple
zeta-functions (of the generalized Euler-Zagier type), under 
the motivation of finding suitable rigorous meaning
of the values of multiple zeta-functions at non-positive integer points.
We reveal that multiple zeta-functions (which are known to be 
meromorphic in the whole space
with infinitely many singular hyperplanes)
turn to be entire on the whole space after taking the desingularization.
The desingularized function is given by
a suitable finite `linear' combination of multiple zeta-functions 
with some arguments shifted.
It is shown that specific combinations of Bernoulli numbers attain the
special values at their non-positive integers of the desingularized
ones.
We also discuss twisted multiple zeta-functions, which can be continued to 
entire functions,
and their special values at non-positive integer points can be explicitly calculated.\\
\end{abstract} 

\setcounter{section}{-1}
\section{Introduction} \label{sec-1}


We begin with the {\bf multiple zeta-function of 
the generalized Euler-Zagier type} defined by
\begin{align}
&\zeta_r((s_j);(\gamma_j))= 
\zeta_r(s_1,\ldots,s_r;\gamma_1,\ldots,\gamma_r):=\sum_{\substack{m_1=1}}^\infty\cdots \sum_{\substack{m_r=1}}^\infty
    \prod_{j=1}^r
    \left(m_1\gamma_1+\cdots+m_j\gamma_j\right)^{-s_j}   
\label{gene-EZ}
\end{align}
for complex variables $s_1,\ldots,s_r$, where
$\gamma_1,\ldots,\gamma_r$ are complex parameters whose real parts are all positive. 
Series \eqref{gene-EZ} converges absolutely 
in the region
\begin{equation}                                                                      
\dd_r=\{(s_1,\ldots,s_r)\in \mathbb{C}^r~|~\Re (s_{r-k+1}+\cdots+s_r)>k\ 
(1\leqslant k\leqslant r)\}. \label{region-Z}                                         
\end{equation}
The first work which established the meromorphic continuation of 
\eqref{gene-EZ} is Essouabri's thesis \cite{Ess}. 
The third-named author \cite[Theorem 1]{MaJNT} showed that \eqref{gene-EZ} can be continued meromorphically to the whole complex space with infinitely many 
(possible) 
singular hyperplanes.   

A special case of \eqref{gene-EZ} is the {\bf multiple zeta-function of 
Euler-Zagier type} defined by

\begin{equation}                                                                        
\zeta_r((s_j))=
\zeta_r(s_1,s_2,\ldots,s_r)=\sum_{m_1,\ldots,m_r=1}^\infty \prod_{j=1}^{r}
\left(m_1+\cdots+m_j\right)^{-s_j}, \label{MZF-def}                              
\end{equation}
which is absolutely convergent in $\dd_r$.

Note that $\zeta_r((s_j))=\zeta_r((s_j);(1)).$
Its special value $\zeta_r(n_1,\dots,n_r)$ when $n_1,\dots,n_r$ are positive integers 
makes sense when $n_r>1$.
It is called the multiple zeta value (abbreviated as MZV),
history of whose study goes back to the work of Euler \cite{Eu} published in 1776
\footnote{                                                                              
You can find several literatures which cite the paper saying as if it were published 
in 1775.                                                                                
But according to Euler archive                                                          
{\tt http://eulerarchive.maa.org/},                                                     
it was written in 1771, presented in 1775 and published in 1776.                        
}.
For a couple of these decades, it has been intensively studied
in various fields including number theory, algebraic geometry,
low dimensional topology
and mathematical physics.

On the other hand, in the late 1990s, several authors 
investigated its analytic properties, 
though their results have not been published (for the details, see the survey article \cite{M2010}). 
In the early 2000s, Zhao \cite{Zh2000} and Akiyama, Egami and Tanigawa \cite{AET} independently showed that \eqref{MZF-def} can be meromorphically continued to $\mathbb{C}^r$. Furthermore, 
the \textit{exact} locations of singularities of \eqref{MZF-def} were explicitly determined in
\cite{AET}:
$\zeta_r((s_j))$ for $r\geq 2$ has infinitely many
singular hyperplanes
\begin{align}
& s_r=1,\quad s_{r-1}+s_{r}=2,1,0,-2,-4,-6, \ldots, \notag\\
& s_{r-k+1}+s_{r-k+2}+\cdots+s_r=k-n\quad (3\leqslant k\leqslant r,\ n\in \mathbb{N}_0).\label{EZ-sing}
\end{align}
It is natural to ask how is the behavior of $\zeta_r(-n_1,\dots,-n_r)$ when 
$n_1,\dots,n_r$ are positive (or non-negative) integers.
However, unfortunately, 
almost all non-positive integer points lie on the above singular hyperplanes, so
they are points of indeterminacy. 
For example, according to \cite{AET,AT},
\begin{align*}
\lim_{\varepsilon_1\to 0}\lim_{\varepsilon_2\to 0}\zeta_2(\varepsilon_1,\varepsilon_2)&=\frac{1}{3},\\
\lim_{\varepsilon_2\to 0}\lim_{\varepsilon_1\to 0}\zeta_2(\varepsilon_1,\varepsilon_2)&=\frac{5}{12},\\
\lim_{\varepsilon\to 0}\zeta_2(\varepsilon,\varepsilon)&=\frac{3}{8}.
\end{align*}
There are some other explicit formulas for the values at those
non-positive integer points as limit values when the way of  
approaching to those points are fixed (\cite{AET,AT,Sak09a,Sak09b,Ko2010,Onozuka}).

However points of indeterminacy cannot be easily investigated, so 
we can raise the following fundamental problem.
\begin{problem}\label{prob}
Is there any \lq rigorous' way to give a meaning of $\zeta_r(-n_1,\dots,-n_r)$,
without ambiguity of indeterminacy, 
for $n_1,\ldots,n_r\in{\mathbb Z}_{\geqslant 0}$?
\end{problem}

Several approaches to this problem have been done so far.
Guo and Zhang \cite{GZ},  Manchon and Paycha \cite{MP} 
and also Guo, Paycha and Zhang \cite{GPZ}
discussed a kind of renormalization method.
In the present paper we will develop yet another approach,
called the \textit{desingularization}, in Section \ref{c-1-zeta}.
The Riemann zeta-function $\zeta(s)$ is a meromorphic function on the complex plane
$\mathbb{C}$ with a simple and unique pole at $s=1$.   Hence $(s-1)\zeta(s)$ is an 
entire function. This simple fact may be regarded as a technique to resolve a 
singularity of $\zeta(s)$ and yield an entire function.  
Our desingularization method is motivated by this simple observation.
For $r\geqslant 2$, multiple
zeta-functions have infinitely many singular loci. 
We will show that
a suitable \textit{finite} sum of multiple zeta-functions will cause cancellations of all of those singularities to produce an entire function
whose special values at non-positive integers are described explicitly
in terms of Bernoulli numbers
(see {\sc{Figure}} \ref{fig:0} and
\eqref{ex-04-3} for the case $r=2$).

\begin{figure}[h]
  \centering
  \includegraphics[bb=0 0 232 113]{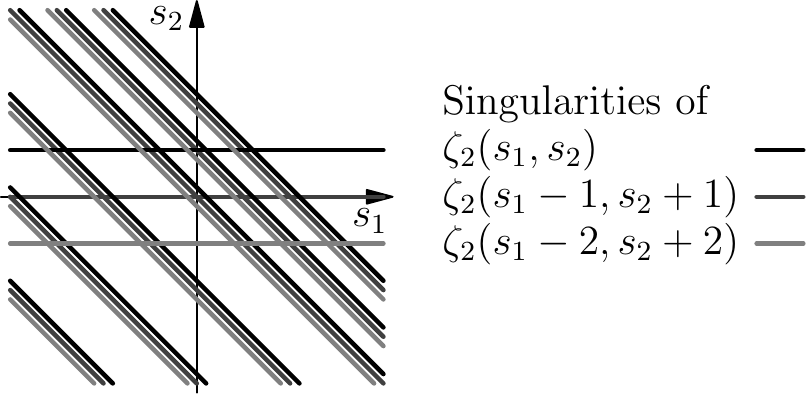}
  \caption{Singularities of $\zeta_2$'s}
  \label{fig:0}
\end{figure}

Another possible approach to the above Problem \ref{prob} is to consider the twisted
multiple series.
Let $\xi_1,\ldots,\xi_r\in \mathbb{C}$ be roots of unity. For 
$\gamma_1,\ldots,\gamma_r\in \mathbb{C}$ with
$\Re \gamma_j >0$ ($1\leqslant j\leqslant r$), 
define the {\bf multiple zeta-function of the generalized Euler-Zagier-Lerch type} by
\begin{equation}                                                                        
\label{Barnes-Lerch}                                                                    
\zeta_r((s_j);(\xi_j);(\gamma_j)):=                                                     
    \sum_{\substack{m_1=1}}^\infty\cdots \sum_{\substack{m_r=1}}^\infty                 
    \prod_{j=1}^r \xi_j^{m_j}
    (m_1\gamma_1+\cdots+m_j\gamma_j )^{-s_j},                               
\end{equation}
which is absolutely convergent in the region $\dd_r$ defined by \eqref{region-Z}. 
We note that the multiple zeta-function of  the generalized Euler-Zagier type 
\eqref{gene-EZ}
is its special case, that is,
$$\zeta_r((s_j);(\gamma_j))=\zeta_r((s_j);(1);(\gamma_j)).$$
Because of the existence of the twisting factor $\xi_1,\ldots,\xi_r$, we can see 
(in Theorem \ref{T-multiple} below) that,
if no $\xi_j$ is equal to $1$, series \eqref{Barnes-Lerch} can be continued to
an entire function, hence its values at non-positive integer points have a rigorous
meaning.    Moreover we will show that those values can be written explicitly in
terms of twisted multiple Bernoulli numbers.

In Section \ref{sec-2-2} we will introduce multiple
twisted Bernoulli numbers, 
which are connected with
multiple zeta-functions of the generalized Euler-Zagier-Lerch type \eqref{Barnes-Lerch}.
After discussing the aforementioned properties of \eqref{Barnes-Lerch} in Section \ref{MZF}, 
we will develop our method of desingularization  
in Section \ref{c-1-zeta}.
Multiple zeta-functions
\eqref{gene-EZ}
are meromorphically continued to the whole space
with their singularities lying on infinitely many hyperplanes.
Our desingularization is a method to reduce them into entire functions (Theorem \ref{T-c-1-zeta}).
We will further show that the desingularized
functions are given by
a suitable finite `linear' combination of multiple zeta-functions
\eqref{gene-EZ} with some arguments shifted (Theorem \ref{Th-ex})
This is the most important result in the present paper, in which we see
a miraculous cancellation of all of their {\it infinitely} many 
singular hyperplanes 
occurring there
by taking a suitable {\it finite} combination of these functions.
We will also prove that certain combinations of Bernoulli numbers attain the
special values at their non-positive integers of the desingularized
functions (Theorem \ref{C-Zr}).
Several explicit examples of desingularization will be given in Section \ref{sec-examples}.

It is to be noted that
these observations on our desingularization method
lead to the construction of $p$-adic multiple $L$-functions 
which will be discussed in a separate paper \cite{FKMT}.

\section{Twisted multiple Bernoulli numbers}\label{sec-2-2}

In this section, we first review the definition of classical Bernoulli numbers and Koblitz' twisted
Bernoulli numbers.
Then we  will introduce twisted multiple Bernoulli numbers,
their multiple analogue, 
and investigate their expression as combinations of twisted Bernoulli numbers.

Let $\mathbb{N}$, $\mathbb{N}_0$, $\mathbb{Z}$, $\mathbb{Q}$, $\mathbb{R}$ and $\mathbb{C}$ be the set of natural numbers, non-negative integers, rational integers, rational numbers, real numbers and complex numbers, respectively. 
For $s\in \mathbb{C}$, denote by $\Re s$ and $\Im s$ the real and the imaginary parts of $s$, respectively.


It is well-known that 
$\zeta(s)$ is a meromorphic function on $\mathbb{C}$ with a simple pole at $s=1$, and satisfies
\begin{equation}
\zeta(1-k)=
\begin{cases} 
-\frac{B_k}{k} & (k\in \mathbb{N}_{>1})\\
-\frac{1}{2}  & (k=1),
\end{cases}
\label{1-1-2}
\end{equation}
where $\{B_n\}$ 
are the Bernoulli numbers
\footnote{                                                                               
or better to be called Seki-Bernoulli numbers, because
Takakazu (Kowa) Seki published the work on these numbers, independently,
before Jakob Bernoulli.
}
defined by 
\begin{align*}
\frac{t}{e^t-1}=\sum_{n=0}^\infty B_n\frac{t^n}{n!},
\end{align*}
(see \cite[Theorem 4.2]{Wa}). 

\begin{definition}[{\cite[p.\,456]{Kob79}}]
For any root of unity $\xi$, we define the {\bf twisted Bernoulli numbers} $\{ \aa_n(\xi)\}$ by
\begin{equation}
\hc(t;\xi)=\frac{1}{1-\xi e^{t}}=\sum_{n=-1}^\infty \aa_n(\xi)\frac{t^n}{n!}, \label{def-tw-Ber}
\end{equation}
where we formally let $(-1)!=1$.
\end{definition}

\begin{remark}
Koblitz \cite{Kob79} generally defined the twisted Bernoulli numbers associated with primitive 
Dirichlet characters. 
The above $\{\aa_n(\xi)\}$ correspond to the trivial character.
\end{remark}

In the case $\xi=1$, we have 
\begin{equation}
\aa_{-1}(1)=-1,\qquad \aa_n(1)=-\frac{B_{n+1}}{n+1}\quad (n\in \mathbb{N}_0). \label{Ber-01}
\end{equation}
In the case $\xi\not=1$, 
we have $\aa_{-1}(\xi)=0$ and 
$\aa_n(\xi)=\frac{1}{1-\xi}H_n\left(\xi^{-1}\right)$ $(n\in \mathbb{N}_0)$, where $\{H_n(\lambda)\}_{n\geqslant 0}$ are what is called the Frobenius-Euler numbers associated with $\lambda$ defined by
$$\frac{1-\lambda}{e^t-\lambda}=\sum_{n=0}^\infty H_n(\lambda)\frac{t^n}{n!}$$
(see Frobenius \cite{Fro}). 
We obtain from \eqref{def-tw-Ber} that $\aa_n(\xi)\in \mathbb{Q}(\xi)$. For example, 
\begin{equation}
\begin{split}
& \aa_0(\xi)=\frac{1}{1-\xi},\quad \aa_1(\xi)=\frac{\xi}{(1-\xi)^2},\quad \aa_2(\xi)=\frac{\xi(\xi+1)}{(1-\xi)^3},\\
& \aa_3(\xi)=\frac{\xi(\xi^2+4\xi+1)}{(1-\xi)^4},\quad \aa_4(\xi)=\frac{\xi(\xi^3+11\xi^2+11\xi+1)}{(1-\xi)^5},\ldots
\end{split}
 \label{TBN-exam}
\end{equation}

Let $\mu_k$ be the group of $k$th roots of unity. 
Using the relation
\begin{equation}
\frac{1}{X-1}-\frac{k}{X^{k}-1}=\sum_{\xi\in \mu_k\atop \xi\not=1}\frac{1}{1-\xi X}\qquad (k\in \mathbb{N}_{>1}) \label{log-der}
\end{equation}
for an indeterminate $X$, 
we obtain the following.

\begin{proposition}
Let $c\in \mathbb{N}_{>1}$. For $n\in \mathbb{N}_0$, 
\begin{equation}
\left(1-c^{n+1}\right)\frac{B_{n+1}}{n+1}=\sum_{\xi^c=1\atop \xi\not=1}\aa_n(\xi). \label{2-0-1}
\end{equation}
\end{proposition}

\begin{remark}
Let $\xi$ be a root of unity. 
As an analogue of \eqref{1-1-2}, it holds that
\begin{equation}
\phi(-k;\xi)=\aa_k(\xi)\quad (k\in \mathbb{N}_0),\label{phi-val}
\end{equation}
where $\phi(s;\xi)$ is the {\bf zeta-function of Lerch type}
defined by the meromorphic continuation of the series
\begin{equation}
\phi(s;\xi)=\sum_{m\geqslant 1}\xi^{m}m^{-s} \qquad (\Re s>1)
\label{Lerch-zeta}
\end{equation}
(cf.\ {\cite[Chapter 2,\,Section 1]{Kob}}).
\end{remark}

We see that \eqref{2-0-1} can also be given from the relation 
\begin{equation}
\left(c^{1-s}-1\right)\zeta(s)=\sum_{\xi^c=1\atop \xi\not=1} \phi(s;\xi). \label{rel-phi}
\end{equation}

Now we define certain multiple analogues of twisted Bernoulli numbers. 

\begin{definition}\label{Def-M-Bern}
Let $r\in \mathbb{N}$, 
$\gamma_1,\ldots,\gamma_r\in \mathbb{C}$ 
and let $\xi_1,\ldots,\xi_r\in \mathbb{C}\setminus\{1\}$ be roots of unity. 
Set 
\begin{align}
    \mathfrak{H}_r  (( t_j );( \xi_j); ( \gamma_j))&:=
    \prod_{j=1}^{r} \mathfrak{H}(\gamma_j (\sum_{k=j}^r t_k);\xi_j)=\prod_{j=1}^{r} \frac{1}{1-\xi_j \exp\left(\gamma_j \sum_{k=j}^r t_k\right)}\label{Def-Hr}
\end{align}
and define
{\bf twisted multiple Bernoulli numbers}
\begin{footnote}{We are not sure which is better, ``twisted multiple'', or ``multiple twisted''. 
But we will skip this problem because
it looks that these two adjectives are ``commutative'' here.}
\end{footnote}$\{\aa(n_1,\ldots,n_r;( \xi_j);( \gamma_j))\}$
by 
\begin{align}
    \mathfrak{H}_r  (( t_j );( \xi_j); ( \gamma_j))
    =\sum_{n_1=0}^\infty
    \cdots
    \sum_{n_r=0}^\infty
    \aa(n_1,\ldots,n_r;( \xi_j);( \gamma_j))
    \frac{t_1^{n_1}}{n_1!}
    \cdots
    \frac{t_r^{n_r}}{n_r!}.
\label{Fro-def-r}
\end{align}
\end{definition}
\begin{remark}
  It is possible to generalize the above definition to the
case when $\xi_r=1$.  In this case, the sum with respect to $n_r$ on
the right-hand side of \eqref{Fro-def-r} is from $-1$ to $\infty$, hence gives a
more natural extension of \eqref{def-tw-Ber}.
\end{remark}

In the case $r=1$, we have $\aa_n(\xi_1)=\aa(n;\xi_1;1)$. Note that 
since $\xi_j\not=1$ $(1\leqslant j\leqslant r)$, we see that 
$\mathfrak{H}_r  (( t_j );( \xi_j); ( \gamma_j))$ is holomorphic around the origin with respect to the parameters $t_1,\dots, t_r$,
hence the singular part does not appear on the right-hand side of \eqref{Fro-def-r}.

We immediately obtain the following from \eqref{def-tw-Ber}, \eqref{Def-Hr} and \eqref{Fro-def-r}.

\begin{proposition}\label{prop-M-Bern}
Let $\gamma_1,\ldots,\gamma_r\in \mathbb{C}$ 
and $\xi_1,\ldots,\xi_r\in \mathbb{C}\setminus\{1\}$ be roots of unity. Then 
$\aa(n_1,\ldots,n_r;( \xi_j);( \gamma_j))$ can be expressed as a polynomial in $\{ \aa_{n}(\xi_j)\,|\,1\leqslant j\leqslant r,~{n\geqslant 0}\}$ and $\{\gamma_1,\ldots,\gamma_r\}$ with $\mathbb{Q}$-coefficients, that is, a rational function in $\{\xi_j\}$ and $\{\gamma_j\}$ with $\mathbb{Q}$-coefficients. 
\end{proposition}

\begin{example}\label{Exam-DH}
We consider the case $r=2$. 
Substituting \eqref{def-tw-Ber} into \eqref{Def-Hr} in the case $r=2$, we have
\begin{align*}
& \mathfrak{H}_2(t_1,t_2;\xi_1,\xi_2;\gamma_1,\gamma_2)=\frac{1}{1-\xi_1 \exp\left(\gamma_1(t_1+t_2)\right)}\frac{1}{1-\xi_2 \exp\left(\gamma_2 t_2\right)}\\
    &=\left(\sum_{m=0}^\infty \aa_m(\xi_1)\frac{\gamma_1^m (t_1+t_2)^m}{m!}\right) \left(\sum_{n=0}^\infty \aa_n(\xi_2)\frac{\gamma_2^n t_2^n}{n!}\right)
    \\
    &=\sum_{m=0}^\infty\sum_{n=0}^\infty \aa_m(\xi_1)\aa_n(\xi_2)\left(\sum_{k,j\geqslant 0 \atop k+j=m}\frac{t_1^k t_2^j}{k!j!}\right)\gamma_1^m\gamma_2^n\frac{t_2^n}{n!}.
\end{align*}
Putting $l=n+j$, we have 
\begin{align*}
\mathfrak{H}_2(t_1,t_2;\xi_1,\xi_2;\gamma_1,\gamma_2)
    &=\sum_{k=0}^\infty\sum_{l=0}^\infty \sum_{j=0}^{l}\binom{l}{j}\aa_{k+j}(\xi_1)\aa_{l-j}(\xi_2)\gamma_1^{k+j}\gamma_2^{l-j}\frac{t_1^k}{k!}\frac{t_2^l}{l!},
\end{align*}
which gives 
\begin{equation}
\begin{split}
\aa(k,l;\xi_1,\xi_2;\gamma_1,\gamma_2)&=\sum_{j=0}^{l}\binom{l}{j}\aa_{k+j}(\xi_1)\aa_{l-j}(\xi_2)\gamma_1^{k+j}\gamma_2^{l-j}\quad (k,l\in \mathbb{N}_0).
\end{split}
\label{Eur-exp1}
\end{equation}
For example, we can obtain from \eqref{TBN-exam} that 
\begin{align*}
&\aa(0,0;\xi_1,\xi_2;\gamma_1,\gamma_2)=\frac{1}{(1-\xi_1)(1-\xi_2)},\quad \aa(1,0;\xi_1,\xi_2;\gamma_1,\gamma_2)=\frac{\xi_1\gamma_1}{(1-\xi_1)^2(1-\xi_2)},\\
&\aa(0,1;\xi_1,\xi_2;\gamma_1,\gamma_2)=\frac{\xi_1\gamma_1+\xi_2\gamma_2-\xi_1\xi_2(\gamma_1+\gamma_2)}{(1-\xi_1)^2(1-\xi_2)^2},\\
&\aa(1,1;\xi_1,\xi_2;\gamma_1,\gamma_2)=\frac{\xi_1^2\gamma_1(\gamma_1-\xi_2(\gamma_1+\gamma_2))+\xi_1\gamma_1(\gamma_1-\xi_2(\gamma_1-\gamma_2))}{(1-\xi_1)^3(1-\xi_2)^2},\ldots
\end{align*}
\end{example}

The following series will be treated in our desingularization method in Section \ref{c-1-zeta}.

\begin{definition}\label{define-tilde-H}
For $c\in \mathbb{R}$ and 
$\gamma_1,\ldots,\gamma_r\in \mathbb{C}$ with $\Re \gamma_j >0 \ (1\leqslant j\leqslant r)$, define
\begin{align}
\widetilde{\mathfrak{H}}_r  (( t_j ); ( \gamma_j);c)
& =\prod_{j=1}^{r} \left( \frac{1}{\exp\left(\gamma_j \sum_{k=j}^r t_k\right)-1}-\frac{c}{\exp\left(c\gamma_j \sum_{k=j}^r t_k\right)-1}\right)\notag\\
& =\prod_{j=1}^{r} \left(\sum_{m=1}^\infty  \left(1-c^m\right)B_m\frac{\left(\gamma_j \sum_{k=j}^r t_k\right)^{m-1}}{m!}\right).\label{def-tilde-H}
\end{align}
In particular when $c\in \mathbb{N}_{>1}$, by use of \eqref{log-der}, we have
\begin{align}
\widetilde{\mathfrak{H}}_r  (( t_j ); ( \gamma_j);c)&=\prod_{j=1}^{r} \sum_{\xi_j^c=1 \atop \xi_j\not=1}\frac{1}{1-\xi_j \exp\left(\gamma_j \sum_{k=j}^r t_k\right)}\notag\\
& =\sum_{\xi_1^c=1 \atop \xi_1\not=1}\cdots \sum_{\xi_r^c=1 \atop \xi_r\not=1}\mathfrak{H}_r  (( t_j );( \xi_j); ( \gamma_j)).\label{tilde-H}
\end{align}
\end{definition}

\begin{remark}
We note that $\widetilde{\mathfrak{H}}_r  (( t_j ); ( \gamma_j);c)$ 
is holomorphic around the origin with respect to the parameters $(t_j)$,
and tends to $0$ as $c \to 1$. 
We also note that 
the Bernoulli numbers appear in the Maclaurin expansion of the limit
\begin{equation*}
\lim_{c\to 1}\frac{1}{(c-1)^r}\widetilde{\mathfrak{H}}_r  (( t_j ); ( \gamma_j);c).\label{limit-H}
\end{equation*}
These are important points in our arguments  on desingularization methods
developed in  Section \ref{c-1-zeta}.
\end{remark}

\begin{example}\label{Exam-DH-2}
Similarly to Example \ref{Exam-DH}, we obtain from
\eqref{def-tilde-H} 
with any $c\in \mathbb{R}$ that
\begin{align}
& \widetilde{\mathfrak{H}}_2  (t_1,t_2; \gamma_1,\gamma_2;c) \notag\\
& =\sum_{k,l=0}^\infty \left\{\sum_{j=0}^{l}\binom{l}{j}\left(1-c^{k+j+1}\right)\left(1-c^{l-j+1}\right)\frac{B_{k+j+1}}{k+j+1}\frac{B_{l-j+1}}{l-j+1}\gamma_1^{k+j}\gamma_2^{l-j}\right\} \frac{t_1^kt_2^l}{k! l!}. \label{Convo-Bern}
\end{align}
Therefore it follows from \eqref{Fro-def-r} and \eqref{tilde-H} that
\begin{equation}
\begin{split}
&\sum_{\xi_1\in \mu_c\atop \xi_1\not=1}\sum_{\xi_2\in \mu_c\atop \xi_2\not=1}\aa(k,l;\xi_1,\xi_2;\gamma_1,\gamma_2)\\
&\quad =\sum_{j=0}^{l}\binom{l}{j}\left(1-c^{k+j+1}\right)\left(1-c^{l-j+1}\right)\frac{B_{k+j+1}}{k+j+1}\frac{B_{l-j+1}}{l-j+1}\gamma_1^{k+j}\gamma_2^{l-j}\quad (k,l\in \mathbb{N}_0)
\end{split}
\label{Eur-exp2}
\end{equation}
for $c\in \mathbb{N}_{>1}$. 
\end{example}


\begin{remark}\label{poly-Bernoulli}
Kaneko \cite{Kaneko1997} defined the poly-Bernoulli numbers $\{B_n^{(k)}\}_{n\in \mathbb{N}_0}$ $(k\in \mathbb{Z})$ by use of the polylogarithm of order $k$. 
Explicit relations between twisted multiple Bernoulli numbers and poly-Bernoulli numbers are not clearly known. 
It is noted that, for example, 
$$  B_l^{(2)} = \sum_{j=0}^l \binom{l}{j} \frac{B_{l-j}B_j}{j+1} \quad (l\in \mathbb{N}_0),$$
which resembles \eqref{Eur-exp1} and \eqref{Eur-exp2}.
\end{remark}

\section{Multiple zeta-functions}\label{MZF}

Corresponding to the twisted multiple Bernoulli numbers
$\{\aa((n_j);( \xi_j);( \gamma_j))\}$ is the
multiple zeta-function of the generalized Euler-Zagier-Lerch type
\eqref{Barnes-Lerch} defined in Introduction, 
which is a multiple analogue of $\phi(s;\xi)$.
This function can be continued analytically to the
whole space and interpolates $\aa((n_j);( \xi_j);( \gamma_j))$ at non-positive integers (Theorem \ref{T-multiple}). 

Assume $\xi_j \neq 1$ $(1\leqslant j \leqslant r)$. 
Using the well-known relation 
\begin{equation*}
  u^{-s}=\frac{1}{\Gamma(s)}\int_0^\infty e^{-ut}{t^{s-1}} dt,
\end{equation*}
we obtain
  \begin{align}
    & \zeta_r((s_j);(\xi_j);(\gamma_j))\notag\\
    &=
    \sum_{\substack{m_1=1}}^\infty\cdots \sum_{\substack{m_r=1}}^\infty
    \Bigl(    
    \prod_{j=1}^r \xi_j^{ m_j}\Bigr) 
    \left(\prod_{k=1}^r
    \frac{1}{\Gamma(s_k)}\right)\int_{[0,\infty)^{r}}
    \prod_{k=1}^r
    \exp(-t_k(\sum_{\substack{j\leqslant k}} m_j\gamma_j)) 
    \prod_{k=1}^r t_k^{s_k-1}dt_k\notag
    \\
    &=
    \left(\prod_{k=1}^r \frac{1}{\Gamma(s_k)}\right)\int_{[0,\infty)^{r}}\prod_{j=1}^r
    \frac{\xi_j \exp(-\gamma_j (\sum_{k=j}^r t_k))}{1-\xi_j \exp(-\gamma_j (\sum_{k=j}^r t_k))}
    \prod_{k=1}^r t_k^{s_k-1}dt_k\notag
    \\
    &=
    \left(\prod_{k=1}^r \frac{1}{(e^{2\pi i s_k}-1)\Gamma(s_k)}\right)\int_{\mathcal{C}^{r}}
    \prod_{j=1}^r 
    \frac{\xi_j \exp(-\gamma_j (\sum_{k=j}^r t_k))}{1-\xi_j \exp(-\gamma_j (\sum_{k=j}^r t_k))}
    \prod_{k=1}^r t_k^{s_k-1}dt_k\notag\\
    &=(-1)^r 
    \left(\prod_{k=1}^r \frac{1}{(e^{2\pi i s_k}-1)\Gamma(s_k)}\right)\int_{\mathcal{C}^{r}}
    \mathfrak{H}_r  (( t_j ),( \xi_j^{-1}), ( \gamma_j))\prod_{k=1}^r t_k^{s_k-1}dt_k,\label{Cont}
  \end{align}
where $\mathcal{C}$ is the Hankel contour, that is, the path consisting of the positive real axis (top side), a circle around the origin of radius $\varepsilon$ (sufficiently small), and the positive real axis (bottom side). 
Note that the third equality holds because we can let $\varepsilon \to 0$ 
on the fourth member of \eqref{Cont}. In fact, 
the integrand of the fourth member is holomorphic around the origin with respect to the parameters $(t_j)$ because of $\xi_j \neq 1$ $(1\leqslant j \leqslant r)$. Here we can easily show that the integral on the last member of \eqref{Cont} is absolutely convergent in a usual manner with respect to the Hankel contour. Hence 
we obtain the following.

\begin{theorem}\label{T-multiple}
Let $\xi_1,\ldots,\xi_r\in \mathbb{C}$ 
be roots of unity and 
$\gamma_1,\ldots,\gamma_r\in \mathbb{C}$ with $\Re \gamma_j >0 \ (1\leqslant j\leqslant r)$. 
Assume that 
\begin{equation}\label{non-unity assumption}
\xi_j\neq 1  \quad\text{ for all } j \ (1\leqslant j\leqslant r).
\end{equation}
Then, with the above notation, 
$\zeta_r((s_j);(\xi_j);(\gamma_j))$ can be analytically continued to $\mathbb{C}^r$ as an entire function in $(s_j)$. For $n_1,\ldots,n_r\in \mathbb{N}_0$, 
\begin{equation}
\zeta_r((-n_j);(\xi_j);(\gamma_j))=(-1)^{r+n_1+\cdots+n_r}\aa((n_j);( \xi_j^{-1});( \gamma_j)). \label{multi-val}
\end{equation}
\end{theorem}

\begin{proof}
Since the contour integral on the right-hand side of \eqref{Cont} is holomorphic for all $(s_k)\in \mathbb{C}^r$, we see that $\zeta_r((s_j);(\xi_j);(\gamma_j))$ can be meromorphically continued to $\mathbb{C}^r$ and its possible singularities are located on hyperplanes $s_k=l_k\in \mathbb{N}$ $(1\leqslant k \leqslant r)$ outside of the region of convergence because $(e^{2\pi i s_k}-1)\Gamma(s_k)$ does not vanish at $s_k\in \mathbb{Z}_{\leqslant 0}$. Furthermore, for $s_k=l_k\in \mathbb{N}$, the integrand of the contour integral with respect to $t_k$ on the last member of \eqref{Cont} is holomorphic around $t_k=0$. Therefore, for $l_k\in \mathbb{N}$, we see that
\begin{align*}
    &\lim_{s_k\to l_k}\int_{\mathcal{C}}
    \mathfrak{H}_r  (( t_j ),( \xi_j^{-1}), ( \gamma_j))
    t_k^{s_k-1}dt_k
    =\int_{C_\varepsilon}
\mathfrak{H}_r  (( t_j ),( \xi_j^{-1}), ( \gamma_j))
    t_k^{l_k-1}dt_k
     =0,
\end{align*}
because of the residue theorem, where $C_\varepsilon=\{\varepsilon e^{i\theta}\,|\,0\leqslant \theta\leqslant 2\pi\}$ for any sufficiently small $\varepsilon$. Consequently this implies that $\zeta_r((s_j);(\xi_j);(\gamma_j))$ has no singularity on $s_k=l_k$, namely $\zeta_r((s_j);(\xi_j);(\gamma_j))$ is entire. 
Finally, substituting \eqref{Fro-def-r} into \eqref{Cont}, setting $(s_j)=(-n_j)$ and 
using
\begin{equation*}
  \lim_{s\to -n}\frac{1}{(e^{2\pi is}-1)\Gamma(s)}=\frac{(-1)^n n!}{2\pi i}\quad (n\in \mathbb{N}_0),
\end{equation*}
we obtain \eqref{multi-val}. 
Thus we complete the proof of Theorem \ref{T-multiple}.
\end{proof}

Such a type of explicit formulas for non-positive integer values of twisted multiple
zeta-functions in several variables was already obtained by 
de Crisenoy \cite{Cr2006} in a much more general context (with real coefficients) by
a quite different method. 
Some partial cases of Theorem \ref{T-multiple} 
are also recovered by the results in 
Matsumoto-Tanigawa \cite{MT2003} and Matsumoto-Tsumura \cite{MaTsAA}.

In \cite[Theorem 1]{MaJNT}, it is shown that the multiple zeta-function $\zeta_r((s_j);(\xi_j);(\gamma_j))$ 
of  the generalized Euler-Zagier-Lerch type \eqref{Barnes-Lerch} with all $\xi_j=1$ 
is meromorphically continued to the whole space $\mathbb C^r$
with \textit{possible} singularities. 
A more general type of multiple zeta-function is 
treated in \cite{Ko2010}, 
where equation \eqref{multi-val} without the assertion of being an entire function 
is shown in the case of 
$\xi_j\neq 1$ for all $j$ 
and the meromorphic continuation of 
$\zeta_r((s_j);(\xi_j);(\gamma_j))$ is also given. 
We stress that in a separate paper \cite{FKMT}, we construct a $p$-adic multiple 
$L$-function which can be regarded 
as a $p$-adic analogue of $\zeta_r((s_j);(\xi_j);(\gamma_j))$.



\begin{remark}\label{Rem-forthcoming}
Without assumption \eqref{non-unity assumption},
it should be noted that \eqref{Cont} does not hold generally, more strictly the third equality on the right-hand side does not hold because the Hankel contours necessarily 
cross the singularities of the integrand. 
\end{remark}

In 
our recent paper \cite{FKMT02}, we will show 
the following necessary and sufficient condition that 
$\zeta_r((s_j);(\xi_j);(\gamma_j))$ is entire, and will determine the exact locations of singularities when it is not entire:

\begin{theorem}\label{Forth-coming} 
Let $\xi_1,\ldots,\xi_r\in \mathbb{C}$ 
be roots of unity and 
$\gamma_1,\ldots,\gamma_r\in \mathbb{C}$ with $\Re \gamma_j >0 \ (1\leqslant j\leqslant r)$. Then 
$\zeta_r((s_j);(\xi_j);(\gamma_j))$ can be entire if and only if the condition \eqref{non-unity assumption} holds. 
When it is not entire, one of the following cases occurs:
\begin{enumerate}[{\rm (i)}]
\item The function $\zeta_r((s_j);(\xi_j);(\gamma_j))$ has infinitely many simple singular hyperplanes when $\xi_j=1$ for some $j$ $(1\leqslant j \leqslant r-1)$.
\item The function $\zeta_r((s_j);(\xi_j);(\gamma_j))$ has a unique simple singular hyperplane $s_r=1$ when $\xi_j\neq 1$ for all $j$ $(1\leqslant j \leqslant r-1)$ and $\xi_r=1$.
\end{enumerate}
\end{theorem}

\section{Desingularization of multiple zeta-functions
}\label{c-1-zeta}
In this section we introduce and develop our method of  desingularization.
In our previous section we saw that 
the multiple zeta-function 
$\zeta_r((s_j);(\gamma_j))$
of  the generalized Euler-Zagier type \eqref{gene-EZ}
is meromorphically  continued to the whole space with `true' singularities 
whilst the multiple zeta function $\zeta_r((s_j);(\xi_j);(\gamma_j))$
of  the generalized Euler-Zagier-Lerch type \eqref{Barnes-Lerch} 
under the non-unity assumption \eqref{non-unity assumption} 
is analytically continued to $\mathbb C^r$ as an entire function.
Our desingularization is a technique to resolve all singularities  of $\zeta_r((s_j);(\gamma_j))$
to produce an entire function  $\zeta^{\rm des}_r((s_j);(\gamma_j))$.
Consider the following expression:
\begin{equation}\label{nonsense equation}
\zeta^{\rm des}_r((s_j);(\gamma_j)):=\lim_{c \to 1}
\frac{1}{(c-1)^r}
\sum_{\xi_1^c=1 \atop \xi_1\not=1}\cdots \sum_{\xi_r^c=1 \atop \xi_r\not=1}
\zeta_r((s_j);(\xi_j);(\gamma_j)).
\end{equation}
This is surely nonsense, because $c\in \mathbb{N}_{>1}$ on the right-hand side.
However, because of the holomorphy of $\zeta_r((s_j);(\xi_j);(\gamma_j))$, we observe that 
the left-hand side is also (at least formally) holomorphic.
Our fundamental idea is symbolized in this primitive expression \eqref{nonsense equation}. 
Our idea is motivated from a very simple observation 
$$(1-s)\zeta(s)=
\lim_{c\to 1}\,\frac{1}{c-1}\,\left(c^{1-s}-1\right)\zeta(s).
$$
Here on  the left-hand side we find  an entire function $(1-s)\zeta(s)$,
which is merely a product of $(1-s)$ and the meromorphic function $\zeta (s)$
with a simple pole at $s=1$.
While 
on the right hand-side, when $c\in \mathbb{N}_{>1}$,  we may associate a decomposition
\begin{equation*}
\frac{1}{c-1}\,\left(c^{1-s}-1\right)\zeta(s)=
\frac{1}{c-1}\sum_{\xi^c=1 \atop \xi\not=1}\phi(s;\xi)
\end{equation*}
into a sum of entire functions $\phi(s;\xi)=\zeta_1(s;\xi;1)$.

Our desingularization method, 
a rigorous mathematical formulation to
give  a meaning of \eqref{nonsense equation}
will be settled in Definition \ref{def-MZF-2}.
An application of desingularization to the Riemann zeta function $\zeta(s)$
is given in Example \ref{Exam-SH}.
We will see in Theorem \ref{T-c-1-zeta}
that our $\zeta^{\rm des}_r((s_j);(\gamma_j))$  is entire on the whole space ${\mathbb C}^r$. 
We stress that $\zeta_r^{\rm des}((s_j);(\gamma_j))$ is worthy of 
an important object from the viewpoint of the analytic theory of 
multiple zeta-functions. In fact, its values at not only all positive 
or all non-positive integer points but also arbitrary integer points are 
fully determined (see Example \ref{Exam-1-33}).  

Theorem \ref{C-Zr} will prove that suitable combinations of Bernoulli numbers attain the 
special values at non-positive integers of $\zeta^{\rm des}_r((s_j);(\gamma_j))$.
Theorem \ref{Th-ex}
will reveal that our desingularized  multiple zeta-function $\zeta^{\rm des}_r((s_j);(\gamma_j))$
is actually given by a finite \lq linear' combination of 
the multiple zeta-function $\zeta_r((s_j+m_j);(\gamma_j))$ 
with some arguments appropriately shifted by $m_j\in{\mathbb Z}$ ($1\leqslant j \leqslant r$).
Example \ref{Z-EZ-double} and Remark \ref{EZ-double}
are our specific observations for double variable case.
%

\begin{definition} \label{def-MZF-2}
For $\gamma_1,\ldots,\gamma_r\in \mathbb{C}$ with $\Re \gamma_j >0 \quad (1\leqslant j\leqslant r)$, the \textbf{desingularized multiple zeta-function},
which we also call the \textbf{desingularization of} $\zeta_r((s_j);(\gamma_j))$,
is defined by
  \begin{align}
    & \zeta^{\rm des}_r((s_j);(\gamma_j)) \notag\\
    &:=\underset{c\in \mathbb{R}\setminus \{1\}}
{\lim_{c\to 1}}\frac{(-1)^r}{(c-1)^r}
    \prod_{k=1}^r \frac{1}{(e^{2\pi i s_k}-1)\Gamma(s_k)}\int_{\mathcal{C}^{r}}
    \widetilde{\mathfrak{H}}_r  (( t_j ); ( \gamma_j);c)\prod_{k=1}^r t_k^{s_k-1}dt_k\label{Cont-2-1} 
\end{align}
for $(s_j)\in \mathbb{C}^r$, 
where $\mathcal{C}$ is the Hankel contour used in \eqref{Cont}. Note that \eqref{Cont-2-1} is well-defined because the convergence of the contour integral and of the limit with respect to $c\to 1$ can be justified from Theorem \ref{T-c-1-zeta} (see below).
\end{definition}

\begin{remark}\label{conceptual idea}
By \eqref{Cont} and \eqref{tilde-H}, we may say that
equation \eqref{Cont-2-1} is a rigorous way to make sense of the nonsense equation
\eqref{nonsense equation}. 
\end{remark}

\begin{example}\label{Exam-SH}
In the case $r=1$, set $(r,\gamma_1)=(1,1)$ in \eqref{Cont-2-1}. Similarly to \cite[Theorem 4.2]{Wa}, 
we can easily see that 
\begin{align}
 \zeta^{\rm des}_1(s;1)& =\lim_{c \to 1}\frac{(-1)}{c-1}\cdot 
    \frac{1}{(e^{2\pi i s}-1)\Gamma(s)}\int_{\mathcal{C}}
\left(\frac{1}{e^t-1}-\frac{c}{e^{ct}-1}\right)t^{s-1}dt\notag\\
& =\lim_{c\to 1}\frac{(-1)}{c-1}\left(\zeta(s)-c\sum_{m=1}^\infty \frac{1}{(cm)^s}\right) \notag\\
         & =\lim_{c\to 1}\frac{(-1)}{c-1}\left(1-c^{1-s}\right)\zeta(s)=(1-s)\zeta(s).\label{def-Z1}
\end{align}
Hence $\zeta^{\rm des}_1(s;1)$ can be analytically continued to $\mathbb{C}$.
As was mentioned in Introduction, this is the "proto-type" of desingularization.
\end{example}

More generally we can prove the following theorem.

\begin{theorem}\label{T-c-1-zeta}
For $\gamma_1,\ldots,\gamma_r\in \mathbb{C}$ with $\Re \gamma_j >0 \quad (1\leqslant j\leqslant r)$, 
\begin{align}
    & \zeta^{\rm des}_r((s_j);(\gamma_j)) \notag\\
    &=
    \prod_{k=1}^r \frac{1}{(e^{2\pi i s_k}-1)\Gamma(s_k)}\int_{\mathcal{C}^{r}}
    \lim_{c \to 1}\frac{(-1)^r}{(c-1)^r}\widetilde{\mathfrak{H}}_r  (( t_j ); ( \gamma_j);c)\prod_{k=1}^r t_k^{s_k-1}dt_k\notag\\
    &=\prod_{k=1}^r \frac{1}{(e^{2\pi i s_k}-1)\Gamma(s_k)}\notag\\
    & \ \times \int_{\mathcal{C}^{r}}
    \prod_{j=1}^{r} \lim_{c \to 1}\frac{(-1)}{c-1}\left( \frac{1}{\exp\left(\gamma_j \sum_{k=j}^r t_k\right)-1}-\frac{c}{\exp\left(c\gamma_j \sum_{k=j}^r t_k\right)-1}\right)\prod_{k=1}^r t_k^{s_k-1}dt_k,\label{Cont-2}
  \end{align}
which can be analytically continued to $\mathbb{C}^r$ as an entire function in $(s_j)$. 
\end{theorem}

For the proof of \eqref{Cont-2}, it is enough to prove that 
if $|c-1|$ is sufficiently small, then there exists a function $F:\,\mathcal{C}^{r} \to \mathbb{R}_{>0}$ independent of $c$ such that 
\begin{gather}
  |(c-1)^{-r}\widetilde{\mathfrak{H}}_r  ((t_j);(\gamma_j);c)|\leqslant F((t_j)) \qquad((t_j) \in \mathcal{C}^{r}), \label{eq-01}\\
  \int_{\mathcal{C}^{r}} F((t_j))\prod_{k=1}^r |t_k^{s_k-1}dt_k|<\infty.\label{eq-02}
\end{gather}
Now we aim to construct $F((t_j))$ which satisfies these conditions. 
Let $\nn(\varepsilon)=\{z\in\mathbb{C}~|~|z|\leqslant \varepsilon\}$ and 
$\sss(\theta)=\{z\in\mathbb{C}~|~|\arg z|\leqslant \theta\}$.

Let $\gamma_1,\ldots,\gamma_r\in \mathbb{C}$ with $\Re \gamma_j >0 \quad (1\leqslant j\leqslant r)$. 
Then the following lemma is obvious.

\begin{lemma}\label{L-1-4-1}
There exist $\varepsilon>0$ and $0<\theta<\pi/2$ such that 
\begin{equation}
  \gamma_j \sum_{k=j}^r t_k\in \nn(1)\cup \sss(\theta) 
\end{equation}
for any $(t_j)\in \mathcal{C}^{r}$, where $\mathcal{C}$ is the Hankel contour involving a circle around the origin of radius $\varepsilon$ (see \eqref{Cont}). 
\end{lemma}

Further we prove the following lemma.

\begin{lemma}\label{L-1-4-2}
Let $c\in \mathbb{R}\setminus \{1\}$ satisfying that $|c-1|$ is sufficiently small. Then there exists a constant $A>0$ independent of $c$ such that 
  \begin{equation}
    |c-1|^{-1}\Bigl|\frac{1}{e^y-1}-\frac{c}{e^{cy}-1}\Bigr|<Ae^{-\Re y/2}
  \end{equation}
for any 
  $y\in \nn(1)\cup \sss(\theta)$.
\end{lemma}

\begin{proof}
It is noted that there exists a constant $C>0$ such that 
  \begin{equation*}
    |c-1|^{-1}\Bigl|\frac{1}{e^y-1}-\frac{c}{e^{cy}-1}\Bigr|<C\quad (y\in \nn(1)),
  \end{equation*}
where we interpret this inequality for $y=0$ as that for $y\to 0$. 
Also, for any $y\in \sss(\theta)\setminus \nn(1)$, we have
  \begin{equation*}
    \begin{split}
      |c-1|^{-1}
      \Bigl|\frac{1}{e^y-1}-\frac{c}{e^{cy}-1}\Bigr|
      &=
      |c-1|^{-1}
      \Bigl|\frac{e^{cy}-ce^y+c-1}{(e^y-1)(e^{cy}-1)}\Bigr|
      \\
      &=
      |c-1|^{-1}
      \Bigl|\frac{e^{cy}-e^y+(1-c)(e^y-1)}{(e^y-1)(e^{cy}-1)}\Bigr|
      \\
      &\leqslant \
      |c-1|^{-1}
      \frac{|e^{cy}-e^y|}{|e^y-1||e^{cy}-1|}
      +\frac{1}{|e^{cy}-1|}.
    \end{split}
  \end{equation*}
Hence it is necessary to estimate 
\begin{equation*}
  |c-1|^{-1}
  \frac{|e^{cy}-e^y|}{|e^y-1||e^{cy}-1|}.
\end{equation*}
We note that
\begin{equation*}
  \begin{split}
    \Bigl|\frac{e^{ay}-1}{a}\Bigr|&=\Bigl|\sum_{j=1}^\infty \frac{a^{j-1}y^j}{j!}\Bigr|
    \leqslant \
    |y|\sum_{l=0}^\infty \frac{|ay|^{l}}{l!}
\leqslant \
    |y|e^{|ay|}.
  \end{split}
\end{equation*}
Since $|y|\leqslant \Re y/\cos\theta$, we have
  \begin{equation*}
    \begin{split}
      |c-1|^{-1}
      \frac{|e^{cy}-e^y|}{|e^y-1||e^{cy}-1|}
      &=
      \frac{1}{|1-e^{-y}||e^{cy}-1|}
      \frac{|e^{(c-1)y}-1|}{|c-1|}
      \\
      &\leqslant
      \frac{1}{|1-e^{-y}||e^{cy}-1|}
      |y|e^{|(c-1)y|}
      \\
      &\leqslant
      \frac{|y|e^{\Re y(|c-1|/\cos\theta)}}{|1-e^{-y}||e^{cy}-1|}.
    \end{split}
  \end{equation*}
Therefore, if $|c-1|$ is sufficiently small, then there exists a constant $A>0$ such that 
  \begin{equation*}
      |c-1|^{-1}
      \frac{|e^{cy}-e^y|}{|e^y-1||e^{cy}-1|}
      \leqslant
      Ae^{-\Re y/2}.
  \end{equation*}
This completes the proof.
\end{proof}

\begin{proof}[Proof of Theorem \ref{T-c-1-zeta}]
With the notation provided in Lemmas \ref{L-1-4-1} and \ref{L-1-4-2}, we set
\begin{equation*}
  \begin{split}
    F((t_j))&=A^r \prod_{j=1}^r \exp\left(-\Re (\gamma_j \sum_{k=j}^r t_k/2)\right)
    =A^r \exp\left(-\sum_{j=1}^r\Re (\gamma_j \sum_{k=j}^r t_k/2)\right)
    \\
    &=A^r \exp\left(-\sum_{k=1}^r\Re (t_k(\sum_{j=1}^k \gamma_j /2))\right)
    =A^r \prod_{k=1}^r\exp\left(-\Re (t_k(\sum_{j=1}^k \gamma_j /2))\right).
  \end{split}
\end{equation*}
Then it is clear that $F((t_j))$ satisfies 
\eqref{eq-01} and \eqref{eq-02}.
Hence, by Lebesgue's convergence theorem we see that \eqref{Cont-2} holds. 

Similarly to the proof of Theorem \ref{T-multiple}, 
since the contour integral on the right-hand side of \eqref{Cont-2} is holomorphic for all $(s_k)\in \mathbb{C}^r$, we see that $\zeta^{\rm des}_r((s_j);(\gamma_j))$ can be meromorphically continued to $\mathbb{C}^r$ and its possible singularities are located on hyperplanes $s_k=l_k\in \mathbb{N}$ $(1\leqslant k \leqslant r)$ outside of the region of convergence because $(e^{2\pi i s_k}-1)\Gamma(s_k)$ does not vanish at $s_k\in \mathbb{Z}_{\leqslant 0}$. Furthermore, for $s_k=l_k\in \mathbb{N}$, the integrand of the contour integral with respect to $t_k$ on the right-hand side of \eqref{Cont-2} is holomorphic around $t_k=0$. Therefore, for $l_k\in \mathbb{N}$, we see that
\begin{align*}
    &\lim_{s_k\to l_k}\int_{\mathcal{C}}
\lim_{c \to 1}\frac{(-1)}{c-1}\left( \frac{1}{\exp\left(\gamma_j \sum_{\nu=j}^r t_\nu\right)-1}-\frac{c}{\exp\left(c\gamma_j \sum_{\nu=j}^r t_\nu\right)-1}\right) t_k^{s_k-1}dt_k\\
    &=-\int_{C_\varepsilon}
\lim_{c \to 1}\frac{1}{c-1}\left( \frac{1}{\exp\left(\gamma_j \sum_{\nu=j}^r t_\nu\right)-1}-\frac{c}{\exp\left(c\gamma_j \sum_{\nu=j}^r t_\nu\right)-1}\right) t_k^{l_k-1}dt_k\\
& =0,
  \end{align*}
because of the residue theorem, where $C_\varepsilon=\{\varepsilon e^{i\theta}\,|\,0\leqslant \theta\leqslant 2\pi\}$ for any sufficiently small $\varepsilon$. Consequently this implies that $\zeta^{\rm des}_r((s_j);(\gamma_j))$ has no singularity on $s_k=l_k$, namely $\zeta^{\rm des}_r((s_j);(\gamma_j))$ is entire. 
Thus we complete the proof of Theorem \ref{T-c-1-zeta}.
\end{proof}

\begin{theorem}\label{C-Zr}
For $\gamma_1,\ldots,\gamma_r\in \mathbb{C}$ with $\Re \gamma_j >0 \quad (1\leqslant j\leqslant r)$, 
\begin{equation}
\begin{split}
& \prod_{j=1}^{r} \frac{\left(1-\gamma_j \sum_{k=j}^r t_k\right)\exp\left(\gamma_j \sum_{k=j}^r t_k\right)-1}{\left( \exp\left(\gamma_j \sum_{k=j}^r t_k\right)-1\right)^2} \\
& \quad =\sum_{m_1,\ldots,m_r=0}^\infty (-1)^{m_1+\cdots+m_r}\zeta^{\rm des}_r((-m_j);(\gamma_j))\prod_{j=1}^{r}\frac{t_j^{m_j}}{m_j!}.
\end{split}
\label{cont-gene}
\end{equation}
Hence, for $(k_j)\in \mathbb{N}_0^r$, 
\begin{align}
    \zeta^{\rm des}_r((-k_j);(\gamma_j))& =\prod_{l=1}^r (-1)^{k_l}k_l!\notag\\
    &\quad \times 
\sum_{\nu_{11}\geqslant 0 \atop {\nu_{12},\,\nu_{22}\geqslant 0 \atop {\,\cdots \atop {\nu_{1r},\,\ldots,\,\nu_{rr}\geqslant 0 \atop {\sum_{d=1}^j \nu_{dj}=k_j \atop (1\leqslant j\leqslant r)}}}}}\prod_{j=1}^{r} \left(B_{1+\sum_{l=j}^r\nu_{jl}}\,\gamma_j^{\sum_{l=j}^r \nu_{jl}}\frac{1}{\prod_{d=1}^{j}\nu_{dj}!}\right). \label{Cont-2-02}
  \end{align}
\end{theorem}

\begin{proof}
By \eqref{def-tilde-H}, we have
\begin{align*}
& \lim_{c\to 1}\frac{(-1)^r}{(c-1)^r}\widetilde{\mathfrak{H}}_r  (( t_j )_{j=1}^{r}; ( \gamma_j)_{j=1}^{r};c)\\
&\quad =\lim_{c \to 1}\prod_{j=1}^{r}\frac{(-1)}{c-1}\left( \frac{1}{\exp\left(\gamma_j \sum_{k=j}^r t_k\right)-1}-\frac{c}{\exp\left(c\gamma_j \sum_{k=j}^r t_k\right)-1}\right)\\
& \quad =\prod_{j=1}^{r}\frac{\left(1-\gamma_j \sum_{k=j}^r t_k\right)\exp\left(\gamma_j \sum_{k=j}^r t_k\right)-1}{\left( \exp\left(\gamma_j \sum_{k=j}^r t_k\right)-1\right)^2}.
\end{align*}
Hence we obtain \eqref{cont-gene} from \eqref{Cont-2}. 
Also, by \eqref{def-tilde-H}, we have 
\begin{align*}
& \lim_{c\to 1}\frac{(-1)^r}{(c-1)^r}\widetilde{\mathfrak{H}}_r  (( t_j )_{j=1}^{r}; ( \gamma_j)_{j=1}^{r};c)\\
& =\lim_{c\to 1}\prod_{j=1}^{r} \left(\sum_{m_j=1}^\infty  \frac{c^{m_j}-1}{c-1}B_{m_j}\frac{\left(\gamma_j \sum_{l=j}^r t_l\right)^{m_j-1}}{m_j!}\right)\\
& =\prod_{j=1}^{r} \left(\sum_{m_j=1}^\infty  B_{m_j}\frac{\left(\gamma_j \sum_{l=j}^r t_l\right)^{m_j-1}}{(m_j-1)!}\right)\\
& =\prod_{j=1}^{r} \left(\sum_{n_j=0}^\infty  B_{n_j+1}\gamma_j^{n_j}\sum_{\nu_{jj},\ldots,\nu_{jr}\geqslant 0 \atop \sum_{l=j}^r \nu_{jl}=n_j}\frac{t_j^{\nu_{jj}}}{\nu_{jj}!}\cdots \frac{t_r^{\nu_{jr}}}{\nu_{jr}!}\right)\\
& =\sum_{\nu_{11}\geqslant 0 \atop {\nu_{12},\,\nu_{22}\geqslant 0 \atop {\,\cdots \atop \nu_{1r},\,\nu_{2r},\ldots,\,\nu_{rr}\geqslant 0}}}\prod_{j=1}^{r} \left(B_{1+\sum_{l=j}^r\nu_{jl}}\,\gamma_j^{\sum_{l=j}^r \nu_{jl}}\frac{t_j^{\sum_{d=1}^{j}\nu_{dj}}}{\prod_{d=1}^{j}\nu_{dj}!}\right).
\end{align*}
Hence, substituting the above relation into \eqref{Cont-2} and using the residue theorem with 
$$\lim_{s\to -k}\left(e^{2\pi is}-1\right)\Gamma(s)=\frac{(2\pi i)(-1)^k}{k!}\quad (k\in \mathbb{N}_0),$$
we have
\begin{align*}
    \zeta^{\rm des}_r((-k_j);(\gamma_j))& =\prod_{l=1}^r \frac{(-1)^{k_l}k_l!}{2\pi i}\notag\\
    &\times {(2\pi i)^r}
\sum_{\nu_{11}\geqslant 0 \atop {\nu_{12},\,\nu_{22}\geqslant 0 \atop {\,\cdots \atop {\nu_{1r},\,,\ldots,\,\nu_{rr}\geqslant 0 \atop {\sum_{d=1}^j \nu_{dj}=k_j \atop (1\leqslant j\leqslant r)}}}}}\prod_{j=1}^{r} \left(B_{1+\sum_{l=j}^r\nu_{jl}}\,\gamma_j^{\sum_{l=j}^r \nu_{jl}}\frac{1}{\prod_{d=1}^{j}\nu_{dj}!}\right). 
  \end{align*}
Thus we obtain the assertion.
\end{proof}




Now we give an expression of $\zeta^{\rm des}_r((s_j);(\gamma_j))$ in terms of $\zeta_r((s_j);(1);(\gamma_j))$, which can be regarded as a multiple version of $\zeta^{\rm des}_1(s;1)=(1-s)\zeta(s)$ in the case $r=1$ (see Examples \ref{Exam-SH} and \ref{Ex-ex-1}). 

For $s_j\in \mathbb{C}$ with $\Re s_j>1$ $(1\leqslant j\leqslant r)$ and $\gamma_1,\ldots,\gamma_r\in \mathbb{C}$ with $\Re \gamma_j>0$ $(1\leqslant j \leqslant r)$, we set
\begin{multline}
  I_{c,r}(s_1,\ldots,s_r;\gamma_1,\ldots,\gamma_r):=\frac{1}{\prod_{j=1}^r\Gamma(s_j)}
  \int_{[0,\infty)^r}
  \prod_{j=1}^r dt_j
  \prod_{j=1}^rt_j^{s_j-1}
  \\
  \times\prod_{j=1}^r\Biggl(\frac{1}{\exp\Bigl(\gamma_j\sum_{k=j}^r t_k\Bigr)-1}
  -
  \frac{c}{\exp\Bigl(c\gamma_j\sum_{k=j}^r t_k\Bigr)-1}
  \Biggr). \label{ex-01}
\end{multline}
From Definition \ref{def-MZF-2}, we see that 
$$\zeta^{\rm des}_r((s_j);(\gamma_j))=\lim_{c\to 1}\frac{(-1)^r}{(c-1)^r}I_{c,r}((s_j);(\gamma_j)).$$
For indeterminates $u_j,v_j$ $(1\leqslant j \leqslant r)$, we set 
\begin{equation}
\mathcal{G}((u_j),(v_j)):=\prod_{j=1}^r\Bigl( 1-(u_jv_j+\cdots+u_rv_r)(v_j^{-1}-v_{j-1}^{-1})\Bigr)  \label{def-GG}
\end{equation}
with the
convention $v_0^{-1}=0$, and also define the set of integers $\{ a_{\mb,\nb}\}$ by 
\begin{equation}
  \mathcal{G}((u_j),(v_j))=\sum_{\mb=(l_j) \in \mathbb{N}_0^r \atop {\nb=(m_j) \in \mathbb{Z}^r \atop \sum_{j=1}^r m_j=0}} a_{\mb,\nb}\prod_{j=1}^r u_j^{l_j}v_j^{m_j}, \label{ex-02}
\end{equation}
where the sum on the right-hand side is obviously a finite sum. 
Note that the condition $\sum_{j=1}^r m_j=0$ for the summation indices $\nb=(m_j)$ can be deduced from the fact that the right-hand side of \eqref{def-GG} is a homogeneous polynomial of degree $0$ in $(v_j)$, namely so is that of \eqref{ex-02}. 

\begin{theorem}\label{Th-ex} 
For $\gamma_1,\ldots,\gamma_r\in \mathbb{C}$ with $\Re \gamma_j>0$ $(1\leqslant j \leqslant r)$,
\begin{align}
  \zeta^{\rm des}_r((s_j);(\gamma_j))
& =\sum_{\mb=(l_j) \in \mathbb{N}_0^r \atop {\nb=(m_j) \in \mathbb{Z}^r \atop \sum_{j=1}^r m_j=0}} a_{\mb,\nb}\Bigl(\prod_{j=1}^{r}(s_j)_{l_j}\Bigr)\zeta_r(s_1+m_1,\ldots,s_r+m_r;(1);(\gamma_j))\label{ex-03}
\end{align}
holds for all $(s_j)\in \mathbb{C}^r$, 
where $(s)_0=1$ and $(s)_k=s(s+1)\cdots(s+k-1)$ $(k\in \mathbb{N})$ are the Pochhammer symbols.
\end{theorem}

We emphasize here that
each term of the right-hand side of \eqref{ex-03} is meromorphic
with infinitely many singularities but
taking the above {\it finite} sum of the shifted functions causes \lq miraculous' cancellations 
of all the {\it infinitely} many singularities
to conclude an entire function.

\begin{remark}
In \eqref{ex-03}, the condition $\sum_{j=1}^{r}m_j=0$ implies that all zeta-functions appearing on the both sides have the same \textit{weight} $s_1+\cdots+s_r$. 
\end{remark}


\begin{proof}[Proof of Theorem \ref{Th-ex}]
First we assume that $\Re s_j$ is sufficiently large for $1\leqslant j\leqslant r$. 
From \eqref{Cont} with $(\xi_j)=(1)$, we have
\begin{align}
& \zeta_r((s_j);(1);(\gamma_j))
=\frac{1}{\prod_{j=1}^r\Gamma(s_j)}
\int_{[0,\infty)^r}\prod_{j=1}^r\frac{t_j^{s_j-1}}{\exp\Bigl(\gamma_j
\sum_{k=j}^r t_k\Bigr)-1}\prod_{j=1}^r dt_j. \label{ex-04}
\end{align}
Using the relation
\begin{align*}
&\lim_{c\to1}
\frac{(-1)}{c-1}\Bigl(  \frac{1}{e^y-1}-\frac{c}{e^{cy}-1}\Bigr)\\
& \ =\frac{-1+e^y-ye^y}{(e^y-1)^2}=\frac{1}{e^y-1}-\frac{ye^y}{(e^y-1)^2}\ =E(y)\ (say),
\end{align*}
we have
\begin{multline}
  \zeta^{\rm des}_r((s_j);(\gamma_j))=\lim_{c\to1}\frac{(-1)^r}{(c-1)^r}{I_{c,r}((s_j);(\gamma_j))}
  \\
  \begin{aligned}
    &=
    \lim_{c\to1}\frac{1}{\prod_{j=1}^r\Gamma(s_j)}
    \int_{[0,\infty)^r}
    \prod_{j=1}^r dt_j
    \prod_{j=1}^r t_j^{s_j-1}
    \\
    &
    \qquad\times\prod_{j=1}^r\frac{(-1)}{c-1}\Biggl(\frac{1}{\exp\Bigl(\gamma_j\sum_{k=j}^r t_k\Bigr)-1}
    -
    \frac{c}{\exp\Bigl(c\gamma_j \sum_{k=j}^r t_k\Bigr)-1}
    \Biggr)
    \\
    &=
    \frac{1}{\prod_{j=1}^r\Gamma(s_j)}
    \int_{[0,\infty)^r}
    \prod_{j=1}^r dt_j
    \prod_{j=1}^r t_j^{s_j-1}
\prod_{j=1}^r
E\Bigl(\gamma_j\sum_{k=j}^r t_k\Bigr).
  \end{aligned}
\label{ex-04-02}
\end{multline}
We calculate the last product of \eqref{ex-04-02}. Using the relations 
\begin{align*}
  \frac{1}{e^y-1}&=\sum_{n=1}^\infty e^{-ny}, \quad \frac{e^y}{(e^y-1)^2}=\sum_{n=1}^\infty ne^{-ny},
\end{align*}
we have, for $J\subset\{1,\ldots,r\}$, 
\begin{align}
&  \int_{[0,\infty)^r}
  \prod_{j=1}^r dt_j
  \prod_{j=1}^r t_j^{s_j-1}
  \prod_{j\notin J}\frac{1}{\exp\Bigl(\gamma_j\sum_{k=j}^r t_k\Bigr)-1}
  \prod_{j\in J}\frac{\Bigl(\gamma_j\sum_{k=j}^r t_k\Bigr)\exp\Bigl(\gamma_j\sum_{k=j}^r t_k\Bigr)}{\Bigl(\exp\Bigl(\gamma_j\sum_{k=j}^r t_k\Bigr)-1\Bigr)^2}
\notag\\
&=
  \int_{[0,\infty)^r}
  \prod_{j=1}^r dt_j
  \prod_{j=1}^r t_j^{s_j-1}
  \prod_{j\notin J}\sum_{n_j=1}^\infty\exp\Bigl(-n_j\gamma_j\sum_{k=j}^r t_k\Bigr)\notag\\
& \qquad \times 
  \prod_{j\in J}\sum_{n_j=1}^\infty n_j\exp\Bigl(-n_j\gamma_j\sum_{k=j}^r t_k\Bigr)
  \prod_{j\in J}\Bigl(\gamma_j\sum_{k=j}^r t_k\Bigr)
\notag\\
&=
\sum_{\substack{n_1,\ldots,n_r\geqslant1}}
  \Bigl(\prod_{j\in J}n_j\gamma_j\Bigr)
  \int_{[0,\infty)^r}
  \prod_{j=1}^r dt_j
  \prod_{j=1}^r t_j^{s_j-1}
  \prod_{j=1}^r
  \exp\Bigl(-t_j\sum_{k=1}^j n_k\gamma_k\Bigr)
  \prod_{j\in J}\Bigl(\sum_{k=j}^r t_k\Bigr)
\notag\\
&=
\sum_{\substack{n_1,\ldots,n_r\geqslant 1}}
  \Bigl(\prod_{j\in J}\Bigl(\sum_{k=1}^j n_k\gamma_k-\sum_{k=1}^{j-1} n_k\gamma_{k}\Bigr)\Bigr)\notag\\
& \qquad \times  \int_{[0,\infty)^r}
  \prod_{j=1}^r dt_j
  \prod_{j=1}^r t_j^{s_j-1}
  \prod_{j=1}^r
  \exp\Bigl(-t_j\sum_{k=1}^j n_k\gamma_{k}\Bigr)
  \prod_{j\in J}\Bigl(\sum_{k=j}^r t_k\Bigr)
\notag\\
&=
\sum_{\mb \in \mathbb{N}_0^r}b_{J,\mb}
\sum_{\substack{n_1,\ldots,n_r\geqslant 1}}
  \Bigl(\prod_{j\in J}\Bigl(\sum_{k=1}^j n_k\gamma_k-\sum_{k=1}^{j-1} n_k\gamma_k\Bigr)\Bigr)\notag\\
&\qquad \times  \int_{[0,\infty)^r}
  \prod_{j=1}^r dt_j
  \prod_{j=1}^r t_j^{s_j+l_j-1}
  \prod_{j=1}^r
  \exp\Bigl(-t_j\sum_{k=1}^j n_k\gamma_k\Bigr)
\notag\\
&=
\sum_{\mb\in \mathbb{N}_0^r}b_{J,\mb}
\sum_{\substack{n_1,\ldots,n_r\geqslant 1}}
  \Bigl(\prod_{j\in J}\Bigl(\sum_{k=1}^j n_k\gamma_k-\sum_{k=1}^{j-1} n_k\gamma_k\Bigr)\Bigr)
  \prod_{j=1}^r\Gamma(s_j+l_j)\frac{1}{\Bigl(\sum_{k=1}^j n_k\gamma_k\Bigr)^{s_j+l_j}}
\notag\\
&=
\sum_{\mb\in \mathbb{N}_0^r}b_{J,\mb}
\prod_{j=1}^r\Gamma(s_j+l_j)
\sum_{\substack{n_1,\ldots,n_r\geqslant 1}}
\sum_{K\subset J\setminus\{1\}}(-1)^{|K|}
  \prod_{j=1}^r\frac{1}{\Bigl(\sum_{k=1}^j n_k\gamma_k\Bigr)^{s_j+l_j-\delta_{j\in J\setminus K}-\delta_{j+1\in K}}}
\notag\\
&=
\sum_{\mb\in \mathbb{N}_0^r}b_{J,\mb}
\sum_{K\subset J\setminus\{1\}}(-1)^{|K|}
\Bigl(\prod_{j=1}^r\Gamma(s_j+l_j)\Bigr)
\zeta_r((s_j+l_j-\delta_{j\in J\setminus K}-\delta_{j+1\in K});(1);(\gamma_j)),
\label{ex-05-2}
\end{align}
where $|K|$ implies the number of elements of $K$, 
\begin{equation*}
\delta_{i\in I}=
\begin{cases}
1 & (i\in I)\\
0 & (i\not\in I)
\end{cases}
\end{equation*}
for $I\subset J$, and 
\begin{equation}
  \label{ex-06}
  \prod_{j\in J}\Bigl(\sum_{k=j}^r t_k\Bigr)=\sum_{\mb \in \mathbb{N}_0^r} b_{J,\mb} \prod_{j=1}^rt_j^{l_j}.
\end{equation}
Hence, by \eqref{ex-04-02} we have
\begin{multline}
  \zeta^{\rm des}_r((s_j);(\gamma_j))
  \\
  \begin{aligned}
    &=
    \sum_{J\subset\{1,\ldots,r\}}(-1)^{|J|}
    \sum_{\mb\in \mathbb{N}_0^r}b_{J,\mb}
    \sum_{K\subset J\setminus\{1\}}(-1)^{|K|}
    \Bigl(\prod_{j=1}^r\frac{\Gamma(s_j+l_j)}{\Gamma(s_j)}
    \Bigr)
    \zeta_r((s_j+l_j-\delta_{j\in J\setminus K}-\delta_{j+1\in K});(1);(\gamma_j))
    \\
    &=
    \sum_{J\subset\{1,\ldots,r\}}
    \sum_{K\subset J\setminus\{1\}}(-1)^{|J\setminus K|}
    \sum_{\mb\in \mathbb{N}_0^r}b_{J,\mb}
    \Bigl(\prod_{j=1}^r(s_j)_{l_j}
    \Bigr)
    \zeta_r((s_j+l_j-\delta_{j\in J\setminus K}-\delta_{j+1\in K});(1);(\gamma_j)).
  \end{aligned}
\label{ex-07}
\end{multline}
Finally we set
\begin{equation*}
H((u_j),(v_j)):=
  \sum_{J\subset\{1,\ldots,r\}}
  \sum_{K\subset J\setminus\{1\}}(-1)^{|J\setminus K|}
  \sum_{\mb \in \mathbb{N}_0^r}b_{J,\mb}
  \prod_{j=1}^r
  u_j^{l_j}v_j^{l_j-\delta_{j\in J\setminus K}-\delta_{j+1\in K}}
\end{equation*}
and aim to prove that 
\begin{equation}
\mathcal{G}((u_j),(v_j))=H((u_j),(v_j)). \label{ex-claim}
\end{equation}

It follows from \eqref{ex-06} that
\begin{equation*}
  \begin{split}
    H((u_j),(v_j))
    &=
    \sum_{J\subset\{1,\ldots,r\}}
    \sum_{K\subset J\setminus\{1\}}(-1)^{|J\setminus K|}
    \Bigl(\prod_{j\in J}\sum_{k=j}^r u_kv_k\Bigr)
    \prod_{j=1}^rv_j^{-\delta_{j\in J\setminus K}-\delta_{j+1\in K}}
    \\
    &=
    \sum_{J\subset\{1,\ldots,r\}}
    \Bigl(\prod_{j\in J}\sum_{k=j}^r u_kv_k\Bigr)
    \sum_{K\subset J\setminus\{1\}}
    \prod_{j\in J\setminus K}(-v_j^{-1})
    \prod_{j\in K}v_{j-1}^{-1}.
  \end{split}
\end{equation*}
Since $v_0^{-1}=0$, we have
\begin{equation*}
      \sum_{K\subset J\setminus\{1\}}
    \prod_{j\in J\setminus K}(-v_j^{-1})
    \prod_{j\in K}v_{j-1}^{-1}=\prod_{j\in J}(-v_j^{-1}+v_{j-1}^{-1}).
\end{equation*}
Hence we obtain
\begin{equation}
  \begin{split}
    H((u_j),(v_j))
    &=
    \sum_{J\subset\{1,\ldots,r\}}
    \prod_{j\in J}\Bigl(\sum_{k=j}^r u_kv_k\Bigr)
    (-v_j^{-1}+v_{j-1}^{-1})
    \\
    &=
    \prod_{j=1}^r\Bigl(\Bigl(\sum_{k=j}^r u_kv_k\Bigr)
    (-v_j^{-1}+v_{j-1}^{-1})+1\Bigr)
    \\
    &=
    \prod_{j=1}^r\Bigl(1-\Bigl(\sum_{k=j}^r u_kv_k\Bigr)
    (v_j^{-1}-v_{j-1}^{-1})\Bigr)=\mathcal{G}((u_j),(v_j)).
  \end{split}
\label{ex-08}
\end{equation}
Combining \eqref{ex-02}, \eqref{ex-07} and \eqref{ex-08}, and regarding 
$(s_j)_{l_j}$ and $\zeta_r((s_j+l_j);(1);(\gamma_j))$ as indeterminates $u_j^{l_j}$ and $v_j^{l_j}$, 
we see that \eqref{ex-03} holds when 
$\Re s_j$ is sufficiently large for $1\leqslant j\leqslant r$. 
It is known that each function on the right-hand side can be continued meromorphically to $\mathbb{C}^r$ (see \cite[Theorem 1]{MaJNT}). Since $\zeta^{\rm des}_r((s_j);(\gamma_j))$ is entire, we see that \eqref{ex-03} holds for all $(s_j)\in \mathbb{C}^r$. 
Thus we complete the proof of Theorem \ref{Th-ex}.
\end{proof}

\section{Examples}\label{sec-examples}

\begin{example}\label{Ex-ex-1}
In the case $r=1$ and $\gamma_1=1$, we have
  \begin{equation*}
    \mathcal{G}(u_1,v_1)=1- u_1v_1v_1^{-1}=1-u_1,
  \end{equation*}
namely, $a_{0,0}(1)=1$ and $a_{1,0}(1)=-1$. Hence we have
\begin{equation*}
  \zeta^{\rm des}_1(s;1)=\lim_{c\to1}\frac{(-1)}{c-1}{I_{c,1}(s)}=(s)_0\zeta_1(s;1;1)-(s)_1\zeta_1(s;1;1)=(1-s)\zeta(s),
\end{equation*}
which coincides with \eqref{def-Z1}. We see that 
$$\zeta^{\rm des}_1(1;1)=-1$$ 
and 
$$\zeta^{\rm des}_1(-k;1)=(-1)^{k}B_{k+1} \qquad (k\in \mathbb{N}_0).$$ 
\end{example}

\begin{example}\label{Z-EZ-double}
In the case $r=2$, we can easily check that 
  \begin{equation*}
    \begin{split}
      \mathcal{G}((u_j),(v_j))
      &=(1-(u_1v_1+u_2v_2)v_1^{-1})(1-u_2v_2(v_2^{-1}-v_1^{-1}))
      \\
      &=(1-u_1)(1-u_2)+(u_2^2-u_1u_2)v_1^{-1}v_2-u_2^2v_1^{-2}v_2^2.
    \end{split}
  \end{equation*}
Then \eqref{ex-03} implies that
\begin{align}
  \zeta^{\rm des}_2(s_1,s_2;\gamma_1,\gamma_2)& =(s_1-1)(s_2-1)\zeta_2(s_1,s_2;(1);\gamma_1,\gamma_2)\notag
  \\
  & \quad +s_2(s_2+1-s_1)\zeta_2(s_1-1,s_2+1;(1);\gamma_1,\gamma_2)\notag\\
  & \quad -s_2(s_2+1)\zeta_2(s_1-2,s_2+2;(1);\gamma_1,\gamma_2). \label{ex-04-2}
\end{align}
Let $k,l\in \mathbb{N}_0$. 
By \eqref{Cont-2-02}, 
we obtain
\begin{equation}
\zeta^{\rm des}_2(-k,-l;\gamma_1,\gamma_2)=(-1)^{k+l}\sum_{\nu=0}^{l}\binom{l}{\nu}B_{k+\nu+1}B_{l-\nu+1}\gamma_1^{k+\nu}\gamma_2^{l-\nu}. \label{EZ-val}
\end{equation}
\end{example}

\begin{remark}\label{EZ-double}

Setting $(\gamma_1,\gamma_2)=(1,1)$ in \eqref{ex-04-2}, we obtain
\begin{align}
  \zeta^{\rm des}_2(s_1,s_2;1,1)
  & =(s_1-1)(s_2-1)\zeta_2(s_1,s_2)\notag\\
  & \quad +s_2(s_2+1-s_1)\zeta_2(s_1-1,s_2+1)-s_2(s_2+1)\zeta_2(s_1-2,s_2+2). \label{ex-04-3}
\end{align}
%
From Theorem \ref{T-c-1-zeta}, 
we see that $\zeta^{\rm des}_2(s_1,s_2;1,1)$ on the left-hand side of \eqref{ex-04-3} 
is entire, 
though each double zeta-function (defined by \eqref{MZF-def}) on the right-hand side of \eqref{ex-04-3} has infinitely many singularities (see \eqref{EZ-sing}).
In fact, we can explicitly write $\zeta^{\rm des}_2(-m,-n;1,1)$ in terms of Bernoulli numbers by \eqref{EZ-val}, though the values of $\zeta_2(s_1,s_2)$ at non-positive integers (except for regular points) cannot be determined uniquely because they are irregular singularities (see \cite{AET}).
%
\end{remark}

\begin{example}
In the case $r=3$, we can see that
\begin{align*}
    \mathcal{G}((u_j),(v_j))
    &=(1-(u_1v_1+u_2v_2+u_3v_3)v_1^{-1})(1-(u_2v_2+u_3v_3)(v_2^{-1}-v_1^{-1}))
    \\
    &\qquad\times
    (1-u_3v_3(v_3^{-1}-v_2^{-1}))
    \\
    &=
    -(u_1-1) (u_2-1) (u_3-1)     
    +(u_1-1) (u_2 - u_3) u_3 v_2^{-1}v_3   \\
    &\qquad
    + (u_1-1) u_3^2 v_2^{-2}v_3^2    
    +(u_1 - u_2) u_2 (u_3-1) v_1^{-1}v_2 \\
    &\qquad
    + u_3 (-u_1 + 2 u_2 - u_1 u_2 + u_2^2 + u_1 u_3 - 2 u_2 u_3) v_1^{-1}v_3 \\
    &\qquad
    - u_3^2 (-1 + u_1 - 2 u_2 + u_3) v_1^{-1}v_2^{-1}v_3^2
    + u_3^3 v_1^{-1}v_2^{-2}v_3^3     \\
    &\qquad
    +u_2^2 (u_3-1)v_1^{-2} v_2^2 
    -  u_2 (2 + u_2 - 2 u_3) u_3 v_1^{-2}v_2 v_3 \\
    &\qquad
    + u_3^2 (-1 - 2 u_2 + u_3) v_1^{-2}v_3^2 
    - u_3^3 v_1^{-2}v_2^{-1}v_3^3.
\end{align*}
Therefore we obtain
\begin{align*}
    &\zeta_3^{\rm des}(s_1,s_2,s_3;\gamma_1,\gamma_2,\gamma_3)\\
    &\quad =
    -(s_1-1) (s_2-1) (s_3-1) \zeta_3(s_1,s_2,s_3;(1);\gamma_1,\gamma_2,\gamma_3)\\
    &\qquad
    + (s_1-1) (-1 + s_2 -  s_3) s_3 \zeta_3(s_1,s_2-1,s_3+1;(1);\gamma_1,\gamma_2,\gamma_3)\\
    &\qquad
    + (s_1-1) s_3 (s_3+1) \zeta_3(s_1,s_2-2,s_3+2;(1);\gamma_1,\gamma_2,\gamma_3)\\
    &\qquad
    +(-1 + s_1 - s_2) s_2 (s_3-1) \zeta_3(s_1-1,s_2+1,s_3;(1);\gamma_1,\gamma_2,\gamma_3)\\
    &\qquad
    + s_3 (s_2 - s_1 s_2 + s_2^2 + s_1 s_3 - 2 s_2 s_3) \zeta_3(s_1-1,s_2,s_3+1;(1);\gamma_1,\gamma_2,\gamma_3)\\
    &\qquad
    - s_3 (s_3+1) (1 + s_1 - 2 s_2 + s_3) \zeta_3(s_1-1,s_2-1,s_3+2;(1);\gamma_1,\gamma_2,\gamma_3)\\
    &\qquad
    + s_3 (s_3+1) (s_3+2) \zeta_3(s_1-1,s_2-2,s_3+3;(1);\gamma_1,\gamma_2,\gamma_3)\\
    &\qquad
    +s_2 (s_2+1) (s_3-1) \zeta_3(s_1-2,s_2+2,s_3;(1);\gamma_1,\gamma_2,\gamma_3)\\
    &\qquad
    -  s_2 (1 + s_2 - 2 s_3) s_3 \zeta_3(s_1-2,s_2+1,s_3+1;(1);\gamma_1,\gamma_2,\gamma_3)\\
    &\qquad    
    + s_3 (s_3+1) (1 - 2 s_2 + s_3) \zeta_3(s_1-2,s_2,s_3+2;(1);\gamma_1,\gamma_2,\gamma_3)\\
    &\qquad
    - s_3 (s_3+1) (s_3+2) \zeta_3(s_1-2,s_2-1,s_3+3;(1);\gamma_1,\gamma_2,\gamma_3).
\end{align*}
Let $k,l,m\in \mathbb{N}_0$. 
By \eqref{Cont-2-02}, we have
\begin{align*}
    \zeta^{\rm des}_3(-k,-l,-m;\gamma_1,\gamma_2,\gamma_3)
    &=(-1)^{k+l+m}\sum_{\nu=0}^m
    \sum_{\rho=0}^{m-\nu}\sum_{\kappa=0}^l\binom{l}{\kappa}\binom{m}{\nu\ \rho}
    \\
    &\quad\times
    B_{k+\nu+\kappa+1}B_{l-\kappa+\rho+1}B_{m-\nu-\rho+1}
    \gamma_1^{k+\nu+\kappa+1}\gamma_2^{l-\kappa+\rho+1}\gamma_3^{m-\nu-\rho+1},
\end{align*}
where $\binom{m}{\nu\ \rho}=\frac{m!}{\nu!\, \rho!\, (m-\nu-\rho)!}$.

\end{example}

\begin{remark}\label{rem-general-form}
Our desingularization method in this paper is for multiple zeta-functions of the generalized Euler-Zagier type \eqref{gene-EZ}.
In 
\cite{FKMT02}, we will extend our desingularization method to more general multiple series.
\end{remark}

\begin{remark}\label{AK-xi}
Arakawa and Kaneko \cite{AK} defined an entire function $\xi_k(s)$ associated with poly-Bernoulli numbers $\{B_n^{(k)}\}$ mentioned in Remark \ref{poly-Bernoulli}. It is known that, for example, 
\begin{align*}
& \xi_1(s)=s\zeta(s+1),\\
& \xi_2(s)=-\zeta_2(2,s)+s\zeta_2(1,s+1)+\zeta(2)\zeta(s).
\end{align*}
Comparing these formulas with \eqref{def-Z1} and \eqref{ex-04-2}, and using the well-known formula
$$\zeta_2(0,s)=\sum_{m,n=1}\frac{1}{(m+n)^s}=\sum_{N=2}^\infty \frac{N-1}{N^s}=\zeta(s-1)-\zeta(s),$$
we obtain
\begin{equation*}
\xi_1(s)=-\zeta^{\rm des}_1(s+1;1), 
\end{equation*}
and 
\begin{align*}
&(1-s)\xi_2(s)=\zeta^{\rm des}_2(2,s;1,1)-\zeta^{\rm des}_1(2)\zeta^{\rm des}_1(s;1)-(s+1)\zeta^{\rm des}_1(s+1;1)+s\zeta^{\rm des}_1(s+2;1), \\
&\zeta^{\rm des}_2(2,s;1,1)=(1-s)\xi_2(s)+\xi_1(1)\xi_1(s-1)-(s+1)\xi_1(s)+s\xi_1(s+1). 
\end{align*}
Note that the both sides of the above relations are entire. 
It seems quite interesting if we acquire explicit relations between $\xi_k(s)$ and $\zeta^{\rm des}_k((s_j);(1))$ for any $k\geqslant 3$.
\end{remark}

Related to the Connes-Kreimer renormalization procedure in quantum field theory,
Guo and Zhang \cite{GZ} and  Manchon and Paycha \cite{MP} 
introduced methods using certain Hopf algebras
to give  well-defined special values of the multiple zeta-functions 
at non-positive integers.

\begin{example}
According to their computation table (in loc.cit.),
Guo-Zhang's renormalized value $\zeta_2^{\rm GZ}(0,-2)$
of $\zeta_2(s_1,s_2)$ at its singularity $(s_1,s_2)=(0,-2)$ is
$$\zeta_2^{\rm GZ}(0,-2)=\frac{1}{120},$$
while Manchon-Paycha's value $\zeta_2^{\rm MP}(0,-2)$ is
$$\zeta_2^{\rm MP}(0,-2)=\frac{7}{720}.$$
On the other hand, our desingularized method gives
$$\zeta_2^{\rm des}(0,-2;1,1)=\frac{1}{18},$$
so these three methods give values different from each other.
\end{example}

\begin{question}
Are there any relationships between our desingularization method and
their renormalization methods?
\end{question}

Finally we emphasize that since our $\zeta^{\rm des}_r((s_j);(1))$ is entire,
their special values at integer points which are neither all positive nor all non-positive
are well-determined.
These values might be also worthy to study.
We conclude this paper with the announcement of explicit examples of those values.

\begin{example}\label{Exam-1-33}
We have
\begin{align*}
& \zeta_2^{\rm des}(-1,1;1,1) =\frac{1}{8},\\
& \zeta_2^{\rm des}(-1,4;1,1) =\zeta(3)-\zeta(4),\\
& \zeta_2^{\rm des}(3,-3;1,1)=\frac{3}{4}-\frac{1}{15}\zeta(3),\\
& \zeta_2^{\rm des}(4,-3;1,1)=\frac{1}{2}+\frac{1}{2}\zeta(2)-\frac{1}{10}\zeta(4).
\end{align*}
Also we can give the following examples for non-admissible indices:
\begin{align*}
\zeta_2^{\rm des}(1,1;1,1)&=\frac{1}{2},\\
\zeta_2^{\rm des}(2,1;1,1)&=-\zeta(2)+2\zeta(3),\\
\zeta_2^{\rm des}(3,1;1,1)&=2\zeta(3)-\frac{5}{4}\zeta(4).
\end{align*}
For the details, see \cite{FKMT02}.
\end{example}

\ 

\bibliographystyle{amsplain}

\

\

\ 

\begin{flushleft}
\begin{small}
H. Furusho\\
Graduate School of Mathematics \\
Nagoya University\\
Furo-cho, Chikusa-ku \\
Nagoya 464-8602, Japan\\
{furusho@math.nagoya-u.ac.jp}

\ 

Y. Komori\\
Department of Mathematics \\
Rikkyo University \\
Nishi-Ikebukuro, Toshima-ku\\
Tokyo 171-8501, Japan\\
komori@rikkyo.ac.jp

\ 

K. Matsumoto\\
Graduate School of Mathematics \\
Nagoya University\\
Furo-cho, Chikusa-ku \\
Nagoya 464-8602, Japan\\
kohjimat@math.nagoya-u.ac.jp

\ 

H. Tsumura\\
Department of Mathematics and Information Sciences\\
Tokyo Metropolitan University\\
1-1, Minami-Ohsawa, Hachioji \\
Tokyo 192-0397, Japan\\
tsumura@tmu.ac.jp

\end{small}
\end{flushleft}

\end{document}